\documentclass[12pt,leqno]{article}%

\usepackage{amsfonts}
\usepackage{amsmath, amsthm, amsfonts, amssymb, color}
\usepackage{mathrsfs}
\newtheorem{thm}{Theorem}[section]

\newtheorem{lem}[thm]{Lemma}

\theoremstyle{definition}

\newcommand{\scr}[1]{\mathscr #1}
\definecolor{wco}{rgb}{0.5,0.2,0.3}
\numberwithin{equation}{section} \theoremstyle{remark}

\def\R{\mathbb R}  \def\ff{\frac} \def\ss{\sqrt} \def\B{\mathbf
B}
  
\def\dd{\delta}  \def\vv{\varepsilon} 
\def\<{\langle} \def\>{\rangle}  \def\gg{\gamma}
    
\def\d{\text{\rm{d}}}   
  \def\si{\sigma} 
\def\beg{\begin} \def\beq{\begin{equation}}

\def\e{\text{\rm{e}}}  \def\OO{\Omega}  \def\oo{\omega}
\def\tt{\tilde}
\def\C{\scr C}      
 
 \def\ll{\lambda}
 
\def\E{\mathbb E} 
\def\Q{\mathbb Q}  \def\LL{\Lambda}
\def\B{\scr B}  
\def\to{\rightarrow}\def\U{\scr U}
\def\8{\infty}\def\q{\mathbb Q}\def\K{\mathbf K}

\newcommand{\hsn}[1]{\|#1\|_{\mathrm{HS}}}
\newcommand{\ra}{\rightarrow}
\newcommand{\dis}{\displaystyle}
\newcommand{\wt}{\widetilde}

\def\R{\mathbb R}

\def\C{\mathscr C}

\def\F{\mathscr F}
\def\d{\mathrm{d}}
\def\E{\mathbb E}
\def\p{\mathbb P}
\def\P{\mathbb P}
\def\q{\mathbb Q}

\title{{\bf Harnack  Inequalities for Stochastic (Functional) Differential Equations with Non-Lipschitzian Coefficients}\footnote{Supported in
 part by  Lab. Math. Com. Sys., NNSFC(11131003), FANEDD (No. 200917), SRFDP and the Fundamental Research Funds for the Central Universities.}
}
\author{
{\bf Jinghai Shao$^{a)}$ ,  Feng-Yu Wang$^{a),b)}$ and Chenggui Yuan$^{b)}$}\\
\footnotesize{$^{a)}$School of Mathematical Sciences,
Beijing Normal
University, Beijing 100875, China}\\
 \footnotesize{$^{b)}$Department of Mathematics,
Swansea University, Singleton Park, SA2 8PP, United Kingdom} 
}
\begin{document}

\maketitle
\begin{abstract}   By using coupling arguments,  Harnack type inequalities are established for a class  of stochastic (functional) differential equations  with multiplicative noises and
 non-Lipschitzian coefficients.  To construct the required couplings, two results on existence and uniqueness of solutions on an open domain are presented. \end{abstract} \noindent
 \vskip 2cm

\section{Introduction}
Consider the following stochastic differential equation (SDE):
\begin{equation}\label{1.1}
\d X(t)=\si(t,X(t))\d B(t)+b(t,X(t))\d t,
\end{equation}
where $(B(t))_{t\geq 0}$ is the $d$-dimensional Brownian motion on a
complete filtered probability space $(\Omega, (\F_t)_{t\geq 0},
\mathscr F, \P)$, $\si:[0,\infty)\times\R^d\ra \R^d\otimes \R^d$ and
$b:[0,\infty)\times\R^d\ra \R^d$ are measurable, locally bounded in
the first variable and continuous in the second variable. This
time-dependent stochastic differential equation   has intrinsic
links to non-linear PDEs (cf. \cite{TWWY}) as well as geometry with
time-dependent metric (cf. \cite{Anton}). When the equation has a
unique solution for any initial data $x$, we denote the solution by
$X^x(t)$.   In this paper we aim to investigate Harnack inequalities
for the associated family of
  Markov operators $(P(t))_{t\ge 0}$:
$$P(t) f(x):= \E f(X^x(t)),\ \ t\ge 0, x\in \R^d, f\in \B_b(\R^d),$$ where $\B_b(\R^d)$ is the set of all bounded measurable functions on $\R^d$.

In the recent work \cite{Wan3} the second named author  established  some Harnack-type  inequalities for $P(t)$ under certain ellipticity and semi-Lipschitz conditions. Precisely,
if  there exists an increasing function $K:[0,\infty)\ra \R$ such that
\[\hsn{\si(t,x)-\si(t,y)}^2+2\< b(t,x)-b(t,y),x-y\>\leq K(t)|x-y|^2,\ x,y\in\R^d,\ t\geq 0,\]
and there exists a decreasing function $\lambda:[0,\infty)\ra (0,\infty)$ such that
$$\|\si(t,x)\xi\|\ge \ll(t) |\xi|,\ \ t\ge 0, \xi,x\in \R^d,$$
then for each $T>0$, the log-Harnack inequality
\beq\label{A0} P(T)\log f(y)\leq \log P(T)f(x)+\frac{K(T)|x-y|^2}{2\lambda(T)^2(1-e^{-K(T)T})}, x,y\in \R^d\end{equation} holds for all strictly positive $f\in \B_b(\R^d).$
If, in addition, there exists an increasing function $\delta:[0,\infty)\ra (0,\infty)$ such that almost surely
\[\big|\big(\si(t,x)-\si(t,y)\big)^*(x-y)\big|\leq \delta(t)|x-y|,\ x,y\in \R^d, t\geq 0,\]
then for $p>(1+\frac{\delta(T)}{\lambda(T)})^2$ there exists a positive constant $C(T)$ (see \cite[Theorem 1.1(2)]{Wan3} for expression of this constant) such that the following Harnack inequality with power $p$ holds:
\beq\label{B0}\big(P(T)f(y)\big)^p\leq \big(P(T)f^p(x)\big)e^{C(T)|x-y|^2},\ x,y\in \R^d, \ f\in \mathscr B_b(\R^d).
\end{equation}
This type Harnack inequality  is first introduced   in \cite{Wan1}
for diffusions on Riemannian manifolds, while the log-Harnack
inequality is firstly studied in \cite{RW09, W10} for semi-linear
SPDEs and reflecting diffusion process on Riemannian manifolds
respectively. Both inequalities have been extended and applied  in
the study of various finite- and infinite-dimensional models,  see
\cite{AK,AZ, BGL, ERS, GW,LW,RW,Wan2,Wan3} and references within. In
particular, these inequalities  have been studied in \cite{WY10} for
the stochastic functional differential equations (SFDE)
\beq\label{y1.1}
 \d X(t)= \big\{Z(t,X(t))+a(t,X_t)\big\}\d t +\si(t, X(t))\d B(t),\ \ X_0\in \C,
 \end{equation}
 where    $\C=C([-r_0,0];\R^d)$ for a fixed constant $r_0>0$ is equipped with the uniform norm $\|\cdot\|_\infty$; $X_t\in \C$ is given by $X_t(u)=X(t+u), u\in [-r_0, 0];$
  $ \si:   [0,\infty)\times\R^d \to \R^d\otimes \R^d,$
$Z:  [0,\infty)\times\R^d \to \R^d,$ and
$a:  [0,\infty)\times\C \to \R^d$ are
  measurable, locally bounded in the first variable and continuous in the second variable. Let $X_t^\phi$ be the solution to this equation with $X_0=\phi\in\C$.
In   \cite{WY10}  the log-Harnack inequality of type (\ref{A0}) and the Harnack inequality of type (\ref{B0}) were established for
$$P_t F(\phi):= \E F(X_t^\phi),\ \  t>0, F\in \B_b(\C)$$ provided
     $\si$ is invertible  and for any $T>0$ there exist constants $K_1, K_2\ge 0, K_3>0$ and $K_4\in \R$ such that \beg{enumerate}
\item[$(1)$] $\big|\si(t, \eta(0))^{-1} \{a(t,\xi)-a(t,\eta)\}\big|\le K_1 \|\xi-\eta\|_\infty,\ t\in [0,T], \xi,\eta\in \C;$\item[$(2)$] $\big|(\si(t,x)-\si(t,y))\big|\le K_2 (1\land |x-y|),\  t\in [0,T], x,y\in \R^d;$
    \item[$(3)$] $\big|\si(t,x)^{-1} \big|\le K_3,\  t\ge 0, x\in\R^d;$ \item[$(4)$] $\hsn{|\si(t,x)-\si(t,y)}^2 + 2\<x-y, Z(t,x)-Z(t,y)\>\le K_4 |x-y|^2,\ t\in [0, T], x,y\in \R^d.$\end{enumerate}

\

The aim of this paper is to extend the above mentioned results  to
SDEs and SFDEs with less regular coefficients as considered in  Fang
and Zhang \cite{FZ} (see also \cite{Lan}), where    the existence
and uniqueness of solutions  were investigated. In section 2, we
consider the SDE  case; and in section 3, we consider the SFDE case.
Finally, in section 4  we present two results for  the existence and
uniqueness  of solutions on open domains of SDEs and SFDEs with
non-Lipschitz coefficients, which are crucial for constructions of
couplings in the proof of     Harnack-type inequalities.

\section{ SDE with non-Lipschitzian coefficients}
To characterize the non-Lipschitz regularity of coefficients, we introduce the class
\beq\label{N1} \U:=\bigg\{u\in C^1((0,\infty); [1,\infty)):\  \int_0^1 \ff{\d s}{su(s)}=\infty,\ \ \liminf_{r\downarrow 0} \big\{u(r)+ru'(r)\big\}>0\bigg\}.\end{equation}
Here, the restriction that $u\ge 1$ is more technical than essential, since in applications one may usually replace $u$ by $u\lor 1$ (see condition {\bf (H1)} below).

To ensure the existence and uniqueness of the solution and to
establish  the log-Harnack  inequality, we shall need the following
assumptions:
\begin{itemize}
\item[{\bf (H1)}]\ There exist $u,\tt u\in\U$ with $u'\le 0$ and increasing functions $K,\tt K\in C([0,\infty); (0,\infty))$ such that for all $t\ge 0$ and $x,y\in\R^d$,
\beg{equation*}\beg{split} &\< b(t,x)-b(t,y),x-y\>+\ff 1 2 \hsn{\si(t,x)-\si(t,y)}^2\leq K(t) |x-y|^2u(|x-y|^2)\\
&\hsn{\si(t,x)-\si(t,y)}^2\leq \tt K(t) |x-y|^2\tt
u(|x-y|^2).\end{split}\end{equation*}
\item[{\bf (H2)}]\ There exists a decreasing function $\ll\in C([0,\infty); (0,\infty))$ such that
$$|\si(t,x) y|\ge \ll(t) | y |,\ \ t\ge 0, x, y\in \R^d.$$
\end{itemize}

The log-Harnack inequality we are establishing depends only on
functions $u, K$ and $\ll$,   $\tt K$ and $\tt u$ will be only used
to ensure  the existence of coupling  constructed in the proof. As
in \cite{Wan3}, in order to derive the Harnack inequality with a
power, we need the following additional assumption:

\ \newline {\bf (H3)} There exists an increasing function $\dd\in C([0,\infty); [0,\infty))$ such that
$$|(\si(t,x)-\si(t,y))^*(x-y)|\le \dd(t) |x-y|,\ \ x,y\in \R^d, t>0.$$

\

\begin{thm}\label{thm1} Assume  that {\bf (H1)}  holds. \beg{itemize}
\item[$(1)$] For any initial data $X(0)$, the equation $(\ref{1.1})$ has a unique  solution, and the solution is non-explosive.
\item[$(2)$]  If moreover {\bf (H2)} holds
and
\beq\label{U}\varphi(s):= \int_0^s u(r)\d r\le \gg s u(s)^2,\ \ s >0\end{equation} for some constant $\gg>0$, then
for each $T>0$ and strictly positive $f\in \B_b(\R^d)$,
$$P(T)\log f(y)\leq \log P(T) f(x)+\frac{K(T)\varphi(|x-y|^2)}{\ll(T)(1-\exp[-2K(T) T/\gg])},\quad f\geq 1,\ x,y\in \R^d.
$$
\item[$(3)$] If, additional to conditions in $(2)$,
 {\bf (H3)} holds, then
$$
\big(P(T)f(y)\big)^q\leq P(T)f^q(x)\cdot\exp\bigg[\frac{K(T) \sqrt q(\sqrt q-1)\varphi(|x-y|^2)}{2 \dd(T)\big((\sqrt{q}-1)\ll(T)- \dd(T)\big)\big(1-\exp[-2K(T) T/\gg]\big)}\bigg]
$$holds for $T>0$, for $q> 1+\frac{\dd(T)+2\ll(T)\sqrt{\dd(T)}}{\ll(T)^2}$, $x,\,y\in \R^d$, and $f\in \mathscr B_b^+(\R^d)$, the set of all non-negative elements in $\B_b(\R^d)$.\end{itemize}
\end{thm}
Typical examples for $u\in \U$ satisfying $u'\le 0$ and (\ref{U}) contain $u(s)=  \log (\e\lor s^{-1}), u(s)= \{\log (\e\lor s^{-1}) \}\log\log (\e^\e\lor s^{-1}), \cdots.$

  Although the main idea of the proof is based on   \cite{Wan3},  due to the non-Lipschitzian coefficients we have to overcome additional difficulties for the construction of coupling.
In fact, to show that the coupling we are going to construct is well defined, a new result concerning existence and uniqueness of solutions to SDEs on a domain is addressed in section 4.

\subsection{Construction of the coupling and some estimates}

It is easy to see fromm Theorem \ref{T3.1} that the equation
(\ref{1.1}) has a unique strong solution which is non-explosive (see
 the beginning of the next subsection).
To establish the desired log-Harnack inequality, we modify the  coupling constructed in  \cite{Wan3}.
For fixed  $T>0$ and   $\theta\in (0,2)$, let
\[\xi(t)=\frac{2-\theta}{2K(T)}\Big[1-\e^{\frac{2K(T)}{\gg}(t-T)}\Big], \quad t\in [0,T], \]
then $\xi$ is a smooth and  strictly positive on $[0,T)$ so that
\beq\label{UV} 2-2K(T) \xi(t)+\gg \xi'(t)=\theta, \ t\in
[0,T).\end{equation} For any $x,y\in \R^d$, we construct the
coupling processes $(X(t),Y(t))_{t\ge 0}$ as follows:
\begin{equation}\label{2.3}\beg{cases}
\d X(t)\!=\!\si(t,X(t))\d B(t)+b(t,X(t))\d t,\ \  X_0=x,\\
\d Y(t)\!=\!\si(t,Y(t))\d B(t)\!+\!b(t,Y(t))\d t\!\\
\qquad\quad \  + \frac 1 {\xi(t)}
\si(t,Y(t))\si(t,X(t))^{-1}(X(t)\!-\!Y(t)) u(|X(t)\!-\!Y(t)|^2)\d t,
\ \  Y_0=y.
\end{cases}
\end{equation}
We intend to show that the $Y(t)$ (hence, the coupling process) is well defined up to time $\tau$ and $\tau\le T$,  where
$$\tau:=\inf\{t\ge 0: X(t)=Y(t)\}$$ is the coupling time. To this end, we apply Theorem \ref{T3.1} to
$$D= \{(x',y')\in \R^d\times\R^d:\ x'\ne y'\}.$$ It is easy to verify (\ref{3.2}) from {\bf (H1)}. Then $Y(t)$ is well defined up to time
$\zeta\land\tau$, where
$$\zeta:=  \lim_{n\ra \infty}\zeta_n,\ \text{and}\ \zeta_n:=\inf\{t\in [0,T);\ |Y(t)|\geq n\}.$$ Here and in what follows,  we set $\inf\emptyset =\infty.$

As in \cite{Wan3}, to derive  Harnack-type inequalities, we need to prove  that the coupling is successful before  $\zeta\land T$ under the weighted probability
$\Q:= R(T\!\land\!\tau\!\land\!\zeta)\P$, where
\begin{equation}\label{2.4}
\begin{split}
R(s):=\!\exp\bigg[&\!-\!\int_0^s\!\ff 1 {\xi(t)}\big\<\si(t,X(t))^{-1}(X(t)-Y(t))u(|X(t)-Y(t)|^2),\d B(t)\big\>\\
&-\frac 12\int_0^s\ff 1
{\xi(t)^2}|\si(t,X(t))^{-1}(X(t)-Y(t))|^2u^2(|X(t)-Y(t)|^2)\d
t\bigg],
\end{split}
\end{equation} for $s\in [0,T\!\land\!\zeta\!\land\! \tau)$.
To ensure the existence of the density $R(T\!\land\!\tau\!\land\!\zeta)$ , letting
$$\tau_n=\inf\{t\in [0,\zeta): |X(t)-Y(t)|\ge n^{-1}\},\ \ \ n\ge 1,$$   we  verify that  $(R(s\!\land\! \zeta_n\!\land\!\tau_n))_{s\in [0,T),n\ge 1}$ is uniformly integrable, so that
$$R({T\land \tau\land\zeta}):=
\lim_{n\to\infty} R({(T-n^{-1})\land\tau_n\land\zeta_n})$$ is a well defined probability density due to the martingale convergence theorem.  Then we prove that  $\zeta\land T\ge\tau$ a.s.-$\Q$, so that $\Q= R({\tau})\P$. Both assertions  are ensured by the following lemma.

\begin{lem}\label{lem1}
Assume that the conditions {\bf (H1)} and {\bf (H2)} hold for some
$u$ satisfying $(\ref{U})$. Then\beg{enumerate}
\item[$(1)$] For any $s\in [0,T)$ and $n\ge 1$,
$$ \E [R({s\wedge \tau_n\land\zeta_n})\log R({s\wedge \tau_n\land\zeta_n})]\leq \frac{K(T)\varphi(|x-y|^2)}{\ll(T)^2\,\theta(2-\theta)(1-\exp[-2K(T) T/\gg])}.$$
Consequently, $R({T\wedge \zeta\land \tau}):=\lim_{n\to\infty} R({(T-n^{-1})\wedge \tau_n\land\zeta_n})$ exists as a probability density function of $\P$, and
$$\E\big\{R({T\wedge \zeta\land \tau})\log R({T\wedge \zeta\land \tau})\big\}\le  \frac{K(T)\varphi(|x-y|^2)}{\ll(T)^2\,\theta(2-\theta)(1-\exp[-2K(T) T/\gg])}.   $$
\item[$(2)$]   Let $\mathbb Q=R({T\wedge \zeta\land\tau}) \P$, then $\mathbb Q(\zeta\land T\ge \tau)=1$. Thus, $\mathbb Q=R({\tau}) \P$ and
$$\E\big\{R({\tau})\log R({\tau})\big\}\le  \frac{K(T)\varphi(|x-y|^2)}{\ll(T)^2\,\theta(2-\theta)(1-\exp[-2K(T) T/\gg])}.   $$
\end{enumerate}
\end{lem}

\begin{proof} (1) Let \begin{equation}\label{2.5}
 \widetilde{B}(t)= B(t)+\int_0^t \!\frac 1 {\xi(s)} \si(s,X(s))^{-1}(X(s)\!-\!Y(s)) u(|X(s)-Y(s)|^2)\d s,\quad \  t<T\land \tau\wedge \zeta.
\end{equation}  Then, before time $T\land\tau\land\zeta$, (\ref{2.3})  can be reformulated as
\begin{equation}\label{FF}
\begin{cases}
\d X(t)=\si(t,X(t))\d \wt B(t)+b(t,X(t))\d t-\frac{X(t)-Y(t)}{\xi(t)}u(|X(t)\!-\!Y(t)|^2)\d t,\ X_0=x,\\
\d Y(t)=\si(t,Y(t))\d \wt B(t)+b(t,Y(t))\d t,\ Y_0=y.
\end{cases}
\end{equation} For fixed $s\in [0,T)$ and $n\ge 1$, let   $\vartheta_{n,s}=s\land\tau_n\land\zeta_n$ and $ \Q_{n,s}= R({\vartheta_{n,s}})\P$. Then by the Girsanov theorem, $(\tt B(t))_{t\in [0,\vartheta_{n,s}]}$ is a $d$-dimensional Brownian motion under the probability measure $\Q_{n,s}$. Let $Z(t)=X(t)-Y(t)$.
By the It\^o formula and condition {\bf (H1)}, we obtain
\begin{align*}
\d |Z(t)|^2&=2\Big\< Z(t), b(t,X(t))\!-\!b(t,Y(t))\!-\!\frac{Z(t)u(|Z(t)|^2)}{\xi(t)}\Big\> \d t\!+\!\hsn{\si(t,X(t))\!-\!\si(t,Y(t))}^2\d t\\
&\hspace{1cm}+2\big\< Z(t), (\si(t,X(t))\!-\!\si(t,Y(t)))\d \wt B(t)\big\>\\
&\leq 2\Big(K(T)\!-\ff 1 {\xi(t)}\Big)|Z(t)|^2u(|Z(t)|^2)\d t\\
&\quad +2\big\< Z(t), (\si(t,X(t))\!-\!\si(t,Y(t)))\d \wt
B(t)\big\>,\ \ t\le \vartheta_{n,s}.
\end{align*}
Applying the It\^o formula to $\varphi(|Z(t)|^2)$ and noting that $\varphi''=u'\leq 0$, we derive
\begin{align*}
\d \varphi(|Z(t)|^2)&\leq \d M(t)+ 2\Big(K(T)-\ff 1 {\xi(t)}\Big)|Z(t)|^2u^2(|Z(t)|^2)\d t,\ \  t\le \vartheta_{n,s},
\end{align*} where
$$ M(t):= \int_0^t 2u(|Z_s|^2)\< Z(s), (\si(s,X(s))-\si(s,Y(s)))\d \wt B(s)\>,\ \ t\le\vartheta_{n,s}$$
is a $\Q_{n,s}$-martingale. Thus, by (\ref{U}) and (\ref{UV}),
\begin{equation}\label{2.7}
\begin{split}
\d \frac{\varphi(|Z(t)|^2)}{\xi(t)}&\leq \ff 1 {\xi(t)}\d M(t)+\frac{2K(T)\xi(t)-2}{\xi(t)^2}|Z(t)|^2u^2(|Z(t)|^2)\d t-\frac{\xi'(t)}{\xi(t)^2}\varphi(|Z(t)|^2)\d t\\
   &\leq \ff 1 {\xi(t)}\d M(t)+\frac{|Z(t)|^2u^2(|Z(t)|^2)}{\xi(t)^2}(-2+2K(T) \xi(t)-\gg \xi'(t))\d  t\\
   &=\ff 1 {\xi(t)}\d M(t)-\theta\frac{|Z(t)|^2u^2(|Z(t)|^2)}{\xi(t)^2}\d  t, \ \ t\le \vartheta_{n,s}.
\end{split}
\end{equation}
 Taking the expectation  w.r.t. the probability measure $\Q_{n,s}$ and noting
$(\wt B(t))_{t\in [0,\vartheta_{n,s}]}$ is a Brownian motion under  $\Q_{n,s}$, we get
\begin{equation}\label{2.8}
\E_{\Q_{n,s}}\bigg[\int_0^{\vartheta_{n,s}}\frac{|Z(t)|^2u^2(|Z(t)|^2)}{\xi(t)^2}\d t\bigg]\leq \frac{\varphi(|x-y|^2)}{\theta\,\xi(0)}.
\end{equation}
On the other hand,  it follows from  {\bf (H2)} that
  \begin{align*}
\log R({\vartheta_{n,s}})&\!=\!-\!\int_0^{\vartheta_{n,s}}\ff 1 {\xi(t)}\< \si(t,X(t))^{-1}Z(t)u(|Z(t)|^2), \d \wt B(t)\>\\\
&\quad +\!\frac 12 \int_0^{\vartheta_{n,s}}\!\! \ff{ \big|\si(t,X(t))^{\!-1}Z(t)\big|^2u^2(|Z(t)|^2)}{\xi(t)^2}\d t\\
&\leq\! -\!\int_0^{\vartheta_{n,s}} \ff 1 {\xi(t)}\< \si(t,X(t))^{-1}Z(t)u(|Z(t)|^2), \d \wt B(t)\>\\
&\quad +\frac 1{2\ll(T)^2 }\int_0^{\vartheta_{n,s}} \frac{|Z(t)|^2u^2(|Z(t)|^2)}{\xi(t)^2}\d t.
\end{align*}
Combining with (\ref{2.8}), we arrive at
\begin{equation}\label{2.9}
\E\big [R({\vartheta_{n,s}})\log R({\vartheta_{n,s}})\big]=\E_{\Q_{n,s}}\big[\log R({\vartheta_{n,s}})\big]\leq \frac{\varphi(|x-y|^2)}{2\ll(T)^2\theta\,\xi(0)},\ s\in [0,T),\, n\geq 1.
\end{equation} This implies the desired inequality in (1), and the consequence then follows from   the martingale convergence theorem.

(2)\ Let $\zeta_n^X=\inf\{t\geq 0;\ |X(t)|\geq n\}$. Since $X(t)$ is
non-explosive as mentioned above, $\zeta_n^X\uparrow \infty$
$\p$-a.s. and hence, also $\mathbb Q$-a.s. For $n>m>1$, it follows
from (\ref{2.7}) that
\begin{equation}\label{2.10}
\ff{\Q(\zeta_m^X>s\land\tau_m>\zeta_n)}{\xi(0)}\int_0^{(n\!-\!m)^2}\!\!\!u(s)\d s \leq \E_{\mathbb Q}\bigg[\frac{\varphi(|Z({\vartheta_{n,s}})|^2)}{\xi_{\vartheta_{n,s}}}\bigg]\leq \frac{\varphi(|x-y|^2)}{\xi(0)}.
\end{equation}
Letting first $n\ra \infty$, then $m\ra \infty$, and noting that $u\ge 1$, we obtain $\mathbb Q(\zeta< s\land\tau)=0$ for all $s\in [0,T)$. Therefore, $\mathbb Q(\zeta\ge T\land\tau)=1.$   So, it remains to show that $\Q(\tau\le T)=1$ and according to (1) and (\ref{2.8}),
$$ \E _\Q \int_0^{T\land\tau} \ff{|Z(t)|^2u(|Z(t)|^2)^2}{\xi(t)^2} \d t\le  \frac{K(T)\varphi(|x-y|^2)}{\ll(T)^2\,\theta(2-\theta)(1-\exp[-2K(T) T/\gg])}. $$
Since $\int_0^T \ff1 {\xi(t)^2}\d t=\infty,$  $\tau>T$ implies that
$$\inf_{t\in [0,T)} |Z({t\land\tau})|^2u(|Z({t\land\tau})|^2)^2>0,$$ which yields that
$$\Q(T<\tau) \le \Q\bigg(\int_0^{T\land\tau} \ff{|Z(t)|^2u(|Z(t)|^2)^2}{\xi(t)^2} \d t=\infty\bigg)=0.$$
Combining this with $\Q(\zeta\ge T\land\tau)=1$, we prove (2). \end{proof}
If moreover {\bf (H3)} holds, then we have the following moment estimate on $R(\tau)$, which will be used to prove the Harnack inequality with power.

\begin{lem}\label{lem2}
Assume {\bf (H1)}, {\bf (H2)} and {\bf (H3)} hold. Then for $p:=\frac{c^2\theta^2}{4\dd(T)^2+4\theta \ll(T)\dd(T)}>0$,
\begin{equation}\label{2.11}
\E R({\tau})^{1+p}\leq
\exp\bigg[\frac{(2\dd(T)+\ll(T)\theta)\theta \varphi(|x-y|^2)}{4\dd(T)\xi(0)(2\dd(T)+2\ll(T)\theta)}\bigg].
\end{equation}
\end{lem}

\begin{proof}
By  (\ref{2.7}) and {\bf (H3)},   for any $r>0$ we have
\begin{align*}
&\E_{\Q_{n,s}}\exp\bigg[r\int_0^{\vartheta_{n,s}}\frac{|Z(t)|^2u^2(|Z(t)|^2)}{\xi(t)^2}\d t\bigg]\\
&\leq \exp\bigg[\frac{r\varphi(|x-y|^2)}{\theta\,\xi(0)}\bigg]\E_{\Q_{n,s}}\exp\bigg[\frac{2r}{\theta}\int_0^{s\wedge \tau_n} \frac{u(|Z(t)|^2)}{\xi(t)}\< Z(t),\big(\si(t,X(t))-\si(t,Y(t))\big)\d \wt B(t)\>\bigg]\\
&\leq \exp\bigg[\frac{r\varphi(|x-y|^2)}{\theta\,\xi(0)}\bigg]\bigg(\E_{\Q_{n,s}}\exp\bigg[\frac{8\dd(T)^2r^2}{\theta^2}\int_0^{\vartheta_{n,s}}\frac{|Z(t)|^2u^2(|Z(t)|^2)}{\xi(t)^2}\d t\bigg]\bigg)^{1/2},
\end{align*}
where in the last step we use the inequality
$$\E \e^{M(t)}\leq \big(\E \e^{2\< M\>(t)}\big)^{1/2},$$
for a continuous exponentially integrable martingale $M(t)$, and $\< M\>(t)$ denotes the quadratic variational process corresponding to $M(t)$.
Putting $\dis r=\frac{\theta^2}{8\dd(T)^2}$ such that $\dis r=\frac{8 r^2 \dd(T)^2}{\theta^2}$, we  get
$$
\E_{\Q_{n,s}}\exp\bigg[\frac{\theta^2}{8\dd(T)^2}\int_0^{\vartheta_{n,s}}
\frac{|Z(t)|^2u^2(|Z(t)|^2)}{\xi(t)^2}\d t\bigg]\leq
\exp\bigg[\frac{\theta\varphi(|x-y|^2)}{4\dd(T)^2\xi(0)}\bigg].$$
Due to Lemma \ref{lem1}, we have $\tau\le T\land \zeta, \Q$-a.s.
 By  taking $s=T-n^{-1}$ and letting $n\ra \infty$  in the above inequality, we arrive at
 \begin{equation}\label{2.12}
\E_{\Q}\exp\bigg[\frac{\theta^2}{8\dd(T)^2}\int_0^{\tau} \frac{|Z(t)|^2u^2(|Z(t)|^2)}{\xi(t)^2}\d t\bigg]\leq \exp\bigg[\frac{\theta\varphi(|x-y|^2)}{4\dd(T)^2\xi(0)}\bigg].
\end{equation}
Since for any   continuous $\Q$-martingale $M(t)$
\begin{align*}
&\E_{\Q}\exp\bigg[pM(t)+\ff p 2\< M\>(t) \bigg]\\
&\leq \bigg(\E_{\Q} \exp\bigg[pqM(t)-p^2q^2\< M\>(t)/2\bigg]\bigg)^{1/q}\bigg(\E_{\Q}\exp\bigg[\frac{pq(pq+1)}{2(q-1)}\< M\>(t)\bigg]\bigg)^{(q-1)/q}\\
&\le \bigg(\E_{\Q}\exp\bigg[\frac{pq(pq+1)}{2(q-1)}\< M\>(t)\bigg]\bigg)^{(q-1)/q}, \quad q>1,
\end{align*}
we obtain from  {\bf (H2)} that
\begin{equation*}
\begin{split}
\E R({\tau})^{1+p}&=\E_{\Q} \exp\bigg[-p\int_0^{ \tau} \ff 1 {\xi(t)}\< \si(t,X(t))^{-1}Z(t)u(|Z(t)|^2),\d \wt B(t)\>\\
   &\hspace{1.9cm}+\frac p 2\int_0^{\tau}\ff 1 {\xi(t)^{2}}\big|\si(t,X(t))^{-1}Z(t)u(|Z(t)|^2)\big|^2\d
   t\bigg]\\
&\le
\bigg(\E_{\Q}\exp\bigg[\frac{pq(pq+1)}{2\ll(T)^2(q-1)}\int_0^{\tau}
\frac{|Z(t)|^2u^2(|Z(t)|^2)}{\xi(t)^2}\d
t\bigg]\bigg)^{(q-1)/q}.\end{split}\end{equation*} Taking
$q=1+\sqrt{1+p^{-1}}$ which minimizes $q(pq+1)/(q-1)$, and using the
definition of $p$, we have
$$\frac{pq(pq+1)}{2\ll(T)^2(q-1)}=\frac{(p+\sqrt{p^2+p})^2}{2\ll(T)^2}=\frac{\theta^2}{8\dd(T)^2}, \ \
  \frac{q-1}{q}=\frac{2\dd(T) +\ll(T)\theta}{2\dd(T)+2\ll(T)\theta}.$$ Combining this with   (\ref{2.12}), we complete the proof. \end{proof}

\subsection{Proof of Theorem \ref{thm1}}

  According to Theorem \ref{T3.1} below for $D=\R^d$, {\bf (H1)} implies that (\ref{1.1}) has a unique solution. Since $u$ is decreasing, the first inequality in {\bf (H1)} with $y=0$ implies that for $|x|\ge 1$,
{\small\beq\label{NE1} 2\<b(t,x),x\>\!+\! \hsn{\si(t,x)}^2 \le 2 \<b(t,0),x\> \!+\!\hsn{\si(t,0)}^2 +2 \hsn{\si(t,0)}\hsn{\si(t,x)}\! + \!K(t) |x|^2 u(1).\end{equation} }Moreover, the second inequality in {\bf (H1)} with $y=0$ implies that for $|x|\ge 1$,
\beg{equation*}\beg{split} \hsn{\si(t,x)} &\le \hsn{\si(t,0)} +\sum_{k=1}^{[|x|]} \bigg\|\si\Big(t,\ff{kx}{[|x|]}\Big)- \si\Big(t,\ff{(k-1)x}{[|x|]}\Big)\bigg\|_{\mathrm{HS}}\\
&\le \hsn{\si(t,0)} +2|x| \ss{\tt K(t)}\ss{u(1)}
\end{split}\end{equation*} where $[|x|]$ stands for the integer part of $|x|$. Combining this with (\ref{NE1}) we may find a function $h\in C([0,\infty); (0,\infty))$ such that
$$2\<b(t,x),x\>+ \hsn{\si(t,x)}^2 \le h(t) (1+|x|^2),$$ which implies the non-explosion of $X(t)$ as is well known. Thus, the proof of (1) is finished.

Next, by Lemma \ref{lem1} and the Girsanov theorem,
$$\tt B(t):= B(t)+ \int_0^{t\land \tau} \frac{\si(s,X(s))^{-1}(X(s)-Y(s))}{\xi(s)}u(|X(s)-Y(s)|^2)\d s,\ \ t\ge 0$$
is a $d$-dimensional Brownian motion under the probability measure $\Q$. Then, according to Theorem \ref{thm1}(1), the equation
\beq\label{Y} \d Y(t)= \si(t,Y(t))\d\tt B(t) + b(t,Y(t))\d t,\ \ Y(0)=y\end{equation} has a unique solution for all $t\ge 0.$ Moreover, it is easy to see that $(X(t))_{t\ge 0}$ solves the equation
\beq\label{X} \d X(t) = \si(t,X(t))\d\tt B(t) +b(t,X(t))\d t - \ff{X(t)-Y(t)}{\xi(t)}1_{\{t<\tau\}}\d t,\ \ X(0)=x.\end{equation} Thus, we have extended equation (\ref{FF}) to all $t\ge 0$, which has a global solution $(X(t),Y(t))_{t\ge 0}$ under the probability measure $\Q$, and
$$\tau:=\inf\{t\ge 0: X(t)=Y(t)\}\le T,\ \ \Q\text{-a.s.}$$   Moreover, since the equations (\ref{Y}) and (\ref{X}) coincide for $t\ge \tau$, by the uniqueness of the solution and $X(\tau)=Y(\tau)$, we conclude that $X(T)=Y(T), \Q$-a.s.

Now, by Lemma \ref{lem1}  and   the Young inequality we obtain
  \begin{align*}
P(T)\log f(y)&=\E_{\q}\big[\log f(Y(T))\big]=\E\big[R({\tau})\log f(Y(T))\big]\\
   &\leq \log \E[ f(X(T))]+\E\big[R({\tau})\log R({\tau})\big]\\
   &\leq \log P(T) f(x)+\frac{K(T)\varphi(|x-y|^2)}{\ll(T)\theta(2-\theta)\big(1-\exp[-2K(T) T/\gg]\big)}.
\end{align*}
Taking $\theta=1$, we derive the desired log-Harnack inequality.

Moreover, by the H\"older inequality,   for any $q>1$ we have
$$
\big(P(T)f(y)\big)^q=\big(\E_{\Q}\big[f(Y(T))\big]\big)^q=\big(\E\big[R_{\tau} f(X(T))\big]\big)^q
\leq \big(P(T)f^q(x)\big)\big(\E\big[R_{\tau}^{q/(q-1)}\big]\big)^{q-1}.$$

Setting $\dis q=1+\frac{4\dd(T)^2+4\theta \ll(T)\dd(T)}{\ll(T)^2\theta^2}$ such that
\begin{equation}\label{2.15}
\frac {q}{q-1}=1+p=1+\frac{\ll(T)^2\theta^2}{4\dd(T)^2+4\theta\ll(T)\dd(T)},
\end{equation}
it then follows from  Lemma \ref{lem2}  that
\[\big(P(T)f(y)\big)^q\leq P(T)f^q(x)\cdot\exp\bigg[\frac{2\dd(T)+\ll(T)\theta}{2\ll(T)^2\dd(T)\theta\,\xi(0)}\varphi(|x-y|^2)\bigg].
\]
It is easy to see that for any $q>1+\frac{\dd(T)^2+2\ll(T)\dd(T) }{\ll(T)^2}$,   (\ref{2.15}) holds for
$\dis \theta= \frac{2\dd(T)}{\ll(T)(\sqrt q-1)}$. Therefore, the desired Harnack inequality with power $q$ follows.

\section{ SFDEs with non-Lipschitzian coefficients}

For a fixed $r_0>0$, let $\C: =C([-r_0, 0]; \R^d)$ denote  all continuous functions from  $[-r_0, 0]$ to $\R^d$ endowed with the uniform norm, i.e.
\[\|\phi\|_\infty:=\max_{-r_0\leq s\leq 0}|\phi(s)|,\quad \text{for}\ \phi\in \C.\]
Let $T>r_0$ be fixed, for any $h \in C([-r_0, T]; \R^d)$  and $t\ge 0,$  let $h_t\in \C$ such that
\[ h_t(s):=h(t+s), \ s\in [-r_0,0].\]
Consider the following type of  stochastic functional differential equation
\begin{equation}\label{3.2.1}
\d X(t)=\{b(t,X(t))+a(t, X_t)\}\d t+\bar \si (t,X_t)\d B(t),\quad
X_0 \in \C,
\end{equation}
where  $a:[0,\infty)\times \C\ra \R^d$, $\bar \si: [0,\infty)\times
\C\ra \R^{d}\otimes \R^{d}$ and $b: [0,\infty)\to\R^d$ are
measurable, locally bounded in the first variable and continuous in
the second variable.

According to the proof of Theorem \ref{T4.2} below, we introduce the
following class of functions to characterize  the non-Lipschitz
regularity of the coefficients:
$$\bar{\scr U}:= \bigg\{u\in C^1((0,\infty),[1,\infty)):\ \int_0^1 \ff{\d s}{su(s)}=\infty,\ s\mapsto su(s)\ \text{is\ increasing\ and\ concave}\bigg\}.$$
According to Theorem \ref{T4.2} with $D=\R^d$, the equation
(\ref{3.2.1}) has a unique strong solution provided there exist a
locally bounded function $K: [0,\infty)\to (0,\infty)$ and $u\in
\bar{\U}$ such that \beg{equation}\label{UN}\beg{split} &
2\<b(t,\phi(0))-b(t,\psi(0))+a(t, \phi)-a(t,\psi), \phi(0)-\psi(0)\>
+\hsn{\bar\si(t, \phi)-\bar\si(t,\psi)}^2\\
&\qquad \le
K(t)\|\phi-\psi\|_\infty^2 u(\|\phi-\psi\|_\infty^2),\\
&\hsn{\bar\si(t,\phi)-\bar\si(t,\psi)}^2 \le
K(t)\|\phi-\psi\|_\infty^2u(\|\phi-\psi\|_\infty^2)\end{split}\end{equation}
holds for all $t\ge 0$ and $\phi,\psi\in\C$. Since $su(s)$ is
increasing and concave in $s$, we have $su(s)\le c(1+s)$ for some
constant $c>0.$ Therefore, it is easy to see that the above
conditions also imply the non-explosion of the solution.

Let $X_t^\phi$ be the segment solution to  (\ref{3.2.1}) for
$X_0=\phi$. We aim to establish the Harnack inequality for the
associated Markov operators $(P_t)_{t\ge 0}$:
$$  P_tf(\phi):=\E f(X_t^\phi),\ \ f\in\B_b(\C),
\phi\in\C.$$ As already known in \cite{ERS, WY10}, to establish a
Harnack inequality using coupling method, one has to assume that
$\bar\si(\cdot,\phi)$ depends only on $\phi(0)$; that is,
$\bar\si(t,\phi)=\si(t,\phi(0))$ holds for some $\si:
[0,\infty)\times \R^d\to \R^d\otimes\R^d.$ Therefore, below we will
consider the equation
\begin{equation}\label{5.1}
\d X(t)=\{b(t,X(t))\}+a(t, X_t)\}\d t+  \si (t,X(t))\d B(t),\quad
X_0 \in \C,
\end{equation} where $a:[0,\infty)\times \C\ra \R^d$, $ \si: [0,\infty)\times
\R^d\ra \R^{d}\otimes \R^{d}$ are measurable, locally bounded in the
first variable and continuous in the second variable. We shall make use of the following assumption, which is weaker than (1)-(4) introduced in the end of Section 1 since $u$ might be unbounded.

\paragraph{(A)}\  There exist $   u\in\bar{\U}$   and increasing function $K, K_1,K_2,K_3,K_4\in C([0,\infty); (0,\infty))$ such that for all $t\ge 0$,
\beg{itemize} \item[{\rm (i)}] $ \< b(t,x)-b(t,y),x-y\>+ \ff 1 2
\hsn{\si(t,x)-\si(t,y)}^2  \leq  K_1(t) |x-y|^2   u(|x-y|^2),$\
$x,y\in \R^d$;
\item[{\rm (ii)}]\ $\hsn{\si(t,x)-\si(t,y)}^2\le K(t) |x-y|^2   u(|x-y|^2),$\
$x,y\in \R^d$; \item[{\rm (iii)}]  $  |a(t,\phi)-a(t,\psi)|^2\le
K_2(t) \|\phi-\psi\|_\infty^2 u(\|\phi-\psi\|^2_\infty),$\
$\phi,\psi\in \C$;
\item[{\rm (iv)}] $\|(\si(t,x)-\si(t,y))\si(t,y)^{-1}\|^2\le
K_3(t),\ \|\si(t,x)^{-1}\|^2\le K_4(t)$,\ $x,y\in\R^d$.
\end{itemize}

\

Obviously, {\bf (A)} implies (\ref{UN}) so that  the equation
(\ref{5.1}) has a unique strong solution and the solution is
non-explosive. Let $G(s)=\int_1^s\ff 1 {ru(r)}\d r, s>0.$ It is easy
to see that $G$ is strictly increasing with full range $\R$. Let
\beg{equation*}\beg{split} &C(T, r)= G^{-1}\Big(G(2r^2)
+G\big(4\{K_1( T) +2K_2(T)K_3(T)+32 K( T)\}\big)\Big),\\
&\Phi(T,r) = C(T,r) u(C(T,r)),\ \ \ \ T
>0.\end{split}\end{equation*} Since $G(0):=\lim_{s\downarrow
0}G(s)=-\infty$, we have $C(T, 0)=0$ for any $T> 0.$  So, if
$\lim_{s\downarrow 0} su(s)=0$ then $\Phi(T,0)=0$.  The main result
in this section is the following.

\begin{thm}\label{thm41}
Assume {\bf (A)}.  If $(\ref{U})$ holds for some constant $\gg>0$,
then for $T>0$
\begin{equation*}
\begin{split}
& P_{T+r_0}\log f(\psi)- \log P_{T+r_0}f(\phi)\\
&\le K_4(T)  \Big(\ff{2\gg\varphi(|\phi(0)\!-\!\psi(0)|^2)}{T} + T
\big\{8K_1(T)^2\!+\! 8K_2(T)K_3(T)\!+\!K_2(T)\big\}\Phi(T,\|\phi\!-\!\psi\|_\infty)\Big),
\end{split}
\end{equation*}holds
for all strictly positive $f\in \B_b(\C)$ and $\phi,\, \psi\in \C$.
\end{thm}

The proof is modified from Section 2. But in the present setting we are not able to derive the Harnack inequality with power as in Theorem \ref{thm1}(3).
The reason is that according to the proof of Lemma \ref{lem44} below, to estimate $\E R(\tt\tau)^q$ for $q>0$ one needs upper bounds   of the exponential moments
of $\|Z_t\|_\infty^2u(\|Z_t\|_\infty^2)$, which is however not available.

Let $T>0$ and
$\phi,\psi\in\C$ be fixed.   Combining the
construction  of coupling  in Section 2 for the SDE case with
non-Lipschitz  coefficients and that in \cite{WY10} for the SFDE
case with Lipschitz coefficients, we construct the coupling process
$(X(t),Y(t))$ as follows:
\begin{equation}\label{5.2}
\begin{cases}
\d X(t)=\{b(t,X(t))+a(t, X_t)\}\d t+\si(t, X(t))\d B(t),\ X_0=\phi,\\
\d Y(t)=\{b(t, Y(t))+a(t, X_t)\}\d t+\si(t, Y(t))\d B(t)\\
  \qquad+ \frac{\si(t,Y(t))\si(t, X(t))^{-1}(X(t)-Y(t))}{\tt\xi(t)}1_{[0,T)}(t) u(|X(t)-Y(t)|^2)\d t, \
  Y_0=\psi,
\end{cases}
\end{equation} where
$$\tt\xi(t)= \ff{T-t}{2\gg},\ \ t\in [0,T].$$

 As explained in Subsection 2.1 for
the existence of solution to (\ref{2.3}) using Theorem \ref{T3.1},
due to Theorem \ref{T4.2} and (i) in {\bf (A)}, the equation (\ref{5.2})
has a unique solution up to the time $T\wedge \tt\zeta\wedge
\tt\tau$, where
$$\tt\tau:=\inf\{t>0:\ X(t)=Y(t)\},\ \
 \tt\zeta:=\lim_{n\ra \infty}\tilde \zeta_n;\ \
\tt\zeta_n:=\inf\{t\in[0,\tilde T): \  |Y(t)|\geq n\}.$$ From {\bf
(A)} it is easy to see that $\tt\zeta\ge T.$  If $\tt\tau\le T$, we
set $Y(t)= X(t)$ for $t\ge \tt \tau$ so that   $(X(t),Y(t))$ solves
(\ref{5.2}) for all $t\ge 0$ (this is not true if $\si(t,Y(t))$ is
replaced by $\bar\si(t,Y_t)$ depending on $Y(t+s), s\in [-r_0,0]$).
In particular, $\tt\tau\le T$ implies that $X_{T+r_0}=X_{T+r_0}.$ To
show that $\tt\tau\le T$, we make use of the Girsanov theorem as in
Section 2. Let $Z(t)=X(t)-Y(t)$ and
$$\LL(t):=\frac{
u(|Z(t)|^2)\si(t, X(t))^{-1}Z(t)}{\tt\xi(t)
}\!+\!\si(t,Y(t))^{-1}\big(a(t,X_t)\!-\!a(t,Y_t)\big).$$

We intend to  show that
\begin{equation}\label{4.5} R(s):= \exp\bigg[-\int_0^s\<\LL(t),\d
B(t)\> -\ff 1 2 \int_0^s |\LL(t)|^2\d t\bigg]\end{equation} is a
uniformly integrable martingale for $s\in [0,
T\wedge  \tt\tau),$ so that due to the Girsanov
theorem,
\begin{equation}\label{4.4}
\begin{split}
 \tilde B(s):= B(s)+ \int_0^s  \LL(t)\d t,
\ \ t<T\wedge  \tt \tau
\end{split}
\end{equation}  is a
$d$-dimensional Brownian motion under the probability $\Q:=
R(\tt\tau\land\tilde \zeta \land T)\P.$ To this end, we make use of
the approximation argument as in Section 2.

Define $$\dis \tt\tau_n=\inf\{t\in [0,\tilde{T}); \ |X(t)-Y(t)|\geq
n^{-1}\},\ \ n\ge 1.$$ By the Girsanov theorem, for any $s\in (0,T)$ and $n\ge 1$,  $\{R(t)\}_{t\in [0,
s\land\tt\tau_n\land\tt\zeta_n]}$ is a martingale and $\{\tt
B(t)\}_{t\in [0, s\land\tt\tau_n\land\tt\zeta_n]}$ is a
$d$-dimensional Brownian motion under the probability
$\Q_{s,n}:=R(s\!\wedge\! \tt\zeta_n\!\land\! \tt\tau_n)\P.$

 For $t< T\land \tt\zeta_n\wedge
\tt\tau_n$, rewrite (\ref{5.2}) as $$\begin{cases}
\d X(t) =\{b(t, X(t))+ a(t, X_t)\}\d t+\si(t,X(t))\d \tilde B(t)-\frac{Z(t)}{\tt\xi(t) } u(|Z(t)|^2)\d t\\
        \qquad\qquad-\si(t,X(t))\si(t,Y(t))^{-1}\big(a(t,X_t)-a(t,Y_t)\big)\d t,\  \ X_0=\phi,\\
\d Y(t) =\{b(t,Y(t))+a(t, Y_t)\}\d t+\si(t, Y(t))\d \tilde B(t),\ \
Y_0=\psi.
\end{cases}$$
We have $Z_0=\phi-\psi$ and
\beq\label{ZZ}\beg{split}
\d Z(t)&=\big(\si(t,X(t))\!-\!\si(t,Y(t))\big)\d \tilde B(t)\!+\!
\Big(b(t,X(t))\!-\!b(t,Y(t))\!-\!\frac{  u(|Z(t)|^2)Z(t)}{\tt\xi(t) }\Big)\d t\\
&\qquad+\big\{\si(t,Y(t))- \si(t,X(t))\big\}\si(t,Y(t))^{-1} (a(t,
X_t)-a(t,Y_t))\d t
\end{split} \end{equation} for  $t<T\land\tt\tau_n\land\tt\zeta_n.$

\begin{lem}\label{lem43} Assume (i),  (ii) and (iii) in {\bf (A)}.
  Let $\E_{s,n}$ stands for
taking the expectation w.r.t. the probability measure
$\Q_{s,n}:=R(s\!\wedge\! \tt\zeta_n\!\land\! \tt\tau_n)\P. $ Then
$$
\sup_{n\ge 1, s\in[0, T)}\E_{s,n}\bigg(\sup_{-r_0\leq t\leq
s\wedge \tilde\zeta_n\tt\wedge \tt\tau_n}\!\!\!|Z(t)|^2\bigg)\leq C(T,\|Z_0\|_\infty).
 $$
\end{lem}

\begin{proof} Let $\ell_n(t)= \sup_{-r_0\le r\le t\land \tt\tau_n\land\tt\zeta_n} |Z(r)|^2.$ By the first inequality  (i) and (iii) in {\bf (A)}, (\ref{ZZ}) and using the It\^o formula, we get
\beg{equation}\label{AB0}\beg{split} \d |Z(t)|^2
 \leq & 2\big\< Z(t), (\si(t,X(t))\!-\!\si(t,Y(t)))\d\tilde B(t)
 \big\>\\
 &+2\Big(K_1(t)|Z(t)|^2  u(|Z(t)|^2)+ |Z(t)|\ss{K_2(t)K_3(t) \|Z_t\|_\infty^2u(\|Z_t\|_\infty^2)}\Big)\d
 t\end{split}\end{equation} for $t\le
 s\land\tt\tau_n\land\tt\zeta_n.$ Moreover, according to the Burkholder-Davis-Gundy
 inequality,   for any continuous martingale $M(t)$ one has
$$ \E \sup_{s\in [0,t]} M(s)  \le 2\ss 2 \E\ss{\<M\>(t)},\ \ t\ge 0.$$
 Combining this with (\ref{AB0}) and (ii) in {\bf (A)}, and noting that $su(s)$
is increasing in $s$ so that
$$|Z(t)|^2u(|Z(t)|^2)\le   \|Z_t\|_\infty^2 u(||Z_t\|_\infty^2)\le
\ell_n(t)u(\ell_n(t)),\ \ t\le s\land\tt\tau_n\land\tt\zeta_n,$$ we
obtain
\begin{align*}
\E_{s,n} \ell_n(t) &\le \|Z_0\|_\infty^2 + 8 \E_{s,n}\ss{K(T)}
\bigg( \int_0^t \ell_n(r)^2 u(\ell_n(r))\d r\bigg)^{1/2} +\ff 1 4 \E_{s,n}\ell_n(t)\\
&\quad +\{2K_1(T) +4K_2(T)K_3(T)\}\int_0^t \E_{s,n} \ell_n(r) u(\ell_n(r))\d r \\
&\le \|Z_0\|_\infty^2\!+\!\ff 1 2\E_{s,n} \ell_n(t)\! +\! 2\{K_1(T)
\!+\!2K_2(T)K_3(T) + 32  K(T)\}\int_0^t  \E_{s,n}
\big[\ell_n(r)u(\ell_n(r))\big]\d r.
\end{align*}
Since $su(s)$ is concave in $s$ so that $\E_{s,n}
[\ell_n(r)u(\ell_n(r))]\le \E_{s,n}\ell_n(r) u(\E_{s,n}\ell_n(r)),$
this implies that
$$ \E_{s,n} \ell_n(t) \le 2 \|Z_0\|_\infty^2  + 4\{K_1(T) +2K_2(T)K_3(T)+32K(T)\}\int_0^t
\E_{s,n}\ell_n(r) u(\E_{s,n}\ell_n(r))\d r,\ \ t\le s.$$  Therefore,
the desired estimate follows from the Bihari's inequality.\end{proof}

\begin{lem}\label{lem44}Assume  {\bf (A)}.  If
$(\ref{U})$ holds for some constant $\gg>0$,  then
\beg{equation*}\beg{split} &\sup_{s\in [0,\tilde T), n\geq 1}\E\big[R(s\!\land\! \tt\zeta_n\!\wedge\! \tt\tau_n)\log R(s\!\land \!\tt\zeta_n\!\wedge\! \tt\tau_n)\big]\\
 &\leq K_4(T) \Big(\ff{2\varphi(|Z(0)|^2)}{T} + T
\big\{8K_1(T)^2+ 8K_2(T)K_3(T)+K_2(T)\big\}\Phi(T,\|Z_0\|_\infty)\Big)
 \end{split}
\end{equation*}
\end{lem}

\begin{proof} By the first inequality in {\bf (A2)}, (\ref{ZZ}) and using the It\^o formula,
we obtain \beg{equation*}\beg{split}  \d |Z(t)|^2
 \leq & 2\big\< Z(t), (\si(t,X(t))\!-\!\si(t,Y(t)))\d\tilde B(t)\big\>
 - \ff{2|Z(t)|^2u(|Z(t)|^2)}{\tt\xi(t)} \d t\\
&+2\Big(K_1(t)|Z(t)|^2  u(|Z(t)|^2)+ |Z(t)|\ss{K_2(t)K_3(t)
\|Z_t\|_\infty^2u(\|Z_t\|_\infty^2)}\Big)\d
 t \end{split}\end{equation*} for $t\le
 s\land\tt\tau_n\land\tt\zeta_n.$
So, as in the proof of Lemma \ref{lem1}, there exists a
$\Q_{s,n}$-martingale $M(t)$ such that for $t\le
s\land\tt\tau_n\land\tt\zeta_n$, \beg{equation*}\beg{split}
\d\ff{\varphi(|Z(t)|^2)}{\tt\xi(t)} \le & \d M(t)- \ff{|Z(t)|^2
u^2(|Z(t)|^2)}{\tt\xi(t)^2} \big(2+\gg \xi'(t)\big)\,\d t\\
& + \ff 2 {\tt\xi(t)} \Big(K_1(t)|Z(t)|^2  u(|(Z(t)|^2))\!+\!
|Z(t)|\ss{K_2(t)K_3(t)
\|Z_t\|_\infty^2u(\|Z_t\|_\infty^2)}\Big)\d t \\
\le & \d M(t) \!+\!  \bigg( 4\{K_1(t)^2\!+\! K_2(t)K_3(t)\}
 \|Z_t\|_\infty^2u(\|Z_t\|_\infty^2)\d t
   \!-\!\ff{|Z(t)|^2u^2(|Z(t)|^2)}{2\tt\xi(t)^2} \bigg) \d
t,\end{split}\end{equation*} where in the last step we have used
$u\ge 1$ and $\tt\xi'(t)=-\ff 1 {2\gg}$. Therefore,
\beq\label{CCC}\beg{split}
&\E_{s,n}\int_0^{s\land\tt\tau_n\land\tt\zeta_n}
\ff{|Z(t)|^2u^2(|Z(t)|^2)}{\tt\xi(t)^2}\d t \\
&\le \ff{2\varphi(|Z(0)|^2)}{\tt\xi(0)} +8T\{K_1(T)^2+
K_2(T)K_3(T)\}
 \E_{s,n}\ell_n(T) u(\ell_n(T)).\end{split}\end{equation}   Since by Lemma
\ref{lem43} and  the concavity of $r\mapsto ru(r)$
$$\E_{s,n}\ell_n(T) u(\ell_n(T))\le
C(T,\|Z_0\|_\infty)u(C(T,\|Z_0\|_\infty))=\Phi(T,\|Z_0\|_\infty),$$ combining (\ref{CCC}) with
  Lemma \ref{lem43} and (iv) in {\bf (A)} we arrive at  that  \beg{equation*}\beg{split} & \E
\big[R(s\land\tt\tau_n\land\tt\zeta_n)\log
R(s\land\tt\tau_n\land\tt\zeta_n)\big] =\ff 1 2 \,\E_{s,n}
\int_0^{s\land\tt\tau_n\land\tt\zeta_n} |\LL(t)|^2\d t\\
&=   K_4(T)\E_{s,n}
\int_0^{s\land\tt\tau_n\land\tt\zeta_n}\Big(
\ff{|Z(t)|^2u^2(|Z(t)|^2)}{\tt\xi(t)^2} + K_2(T)
\|Z_t\|_\infty^2u(\|Z_t\|_\infty^2)\Big) \d t\\
&\le K_4(T) \Big(\ff{2\gamma\varphi(|Z(0)|^2)}{T} + T
\big\{8K_1(T)^2+ 8K_2(T)K_3(T)+K_2(T)\big\}\Phi(T,\|Z_0\|_\infty)\Big).\end{split}\end{equation*}
\end{proof}

\beg{proof}[Proof of Theorem \ref{thm41}] As discussed in Section 2
that Lemma \ref{lem44} and (\ref{CCC}) imply that   $\tt\tau\leq
 T\wedge  \tt\zeta$  $\Q$-a.s., where $\Q:= R(\tt
\tau\land T\land\tt \zeta)\P= R(\tt\tau)\P.$ Since by the construction we have $X(t)=Y(t)$ for $t\ge\tt\tau$, this implies that $X_{T+r_0}=Y_{T+r_0}. $  Applying the Young inequality and Lemma
\ref{lem44}, we obtain
 \begin{equation*}
\begin{split}
&P_{T+r_0}\log f(\psi)-\log P_{T\!+r_0} f (\phi)=\E_{\q}\big[\log f(Y_{T+r_0})\big]-\log P_{T +r_0} f (\phi)\\
&=\E\big[R(\tilde\tau)\log f(X_{T+r_0})\big]-\log
\E\big[f(X_{T+r_0})\big]\le
  \E\big[R(\tilde\tau)\log R(\tilde \tau)\big]\\
  &\leq K_4(T) \Big(\ff{2\gg\varphi(|Z(0)|^2)}{T} \!+\! T
\big\{8K_1(T)^2\!+\! 8K_2(T)K_3(T)\!+\!K_2(T)\big\}\Phi(T,\|Z_0\|_\infty)\Big).
\end{split}
\end{equation*}\end{proof}

\section{  Existence and uniqueness of solutions}

There are a lot of literature on the existence and uniqueness  of
SDEs and SFDEs  under non-Lipschitz condition, see e.g. Taniguchi
\cite{Tan92, Tan10} and references therein. In the following two
subsections, for the construction of couplings given in the previous
sections, we present below  two results in this direction for SDEs
and SFDEs on open domains respectively.

\subsection{Stochastic differential equations}

Let $D$ be a non-empty  open domain in $\R^d,$ and let $T>0$ be fixed. Consider the following SDE:
\begin{equation}\label{3.1}
\d X(t)=\si(t,X(t))\d B(t)+b(t,X(t))\d t,
\end{equation}
where $(B(t))_{t\geq 0}$ is the $m$-dimensional Brownian motion on a
complete filtered probability space $(\Omega, (\F_t)_{t\geq 0},
\mathscr F, \P)$, $\si:[0,T)\times D\ra \R^d\otimes \R^m$ and $b:
[0,T)\times D\ra \R^d$ are measurable, locally bounded in the first
variable and continuous in the second variable.

\beg{thm}\label{T3.1} If there exist $u\in \U$, a sequence of compact sets ${\bf K}_n\uparrow D$ and functions $\{\Theta_n\}_{n\ge 1} \in C([0,T); (0,\infty))$ such that for every $n\ge 1$,
\beq\label{3.2} \beg{split} &2\<b(t,x)-b(t,y), x-y\> +\|\si(t,x)-\si(t,y)\|_{\mathrm{HS}}^2\\
 &\le \Theta_n(t)|x-y|^2 u(|x-y|^2),\ \ |x-y|\le 1, x,y\in {\bf K}_n, t\in [0,T).\end{split}\end{equation} Then for any initial data $X(0)\in D$, the equation $(\ref{3.1})$ has a unique solution $X(t)$ up to life time
$$\zeta:= T\land \lim_{n\to\infty} \inf\big\{t\in [0,T):\ X(t)\notin {\bf K}_n\big\},$$ where $\inf\emptyset :=\infty.$ \end{thm}

\beg{proof} For each $n\ge 1$, we may find $h_n\in C^\infty (\R^d)$ with compact support contained in $D$ such that $h_n|_{{\bf K}_n}=1$. Let
$$b_n(t,x)= h_n(x) b(t,x),\ \ \si_n(t,x) = h_n(x)\si(t,x).$$ Then for any $n\ge 1$, $b_n$ and $\si_n$ are bounded   on $[0, \ff{nT}{n+1}]\times \R^d$ and continuous in the second variable.
According to the Skorokhod theorem \cite{SK} (see also \cite[Theorem
0.1]{HS}), the equation \beq\label{3.3} \d X_n(t)=\si_n(t, X_n(t))\d
B(t)+b_n(t, X_n(t))\d t,\ \ X_n( 0)=X_0\end{equation} has a weak
solution for $t\in [0, \ff{nT}{n+1}].$ So, by Yamada-Watanabe
principle \cite{YW}, to prove the existence and uniqueness of the
(strong) solution, we only need to verify the pathwise uniqueness.

Let $X_n(t),\tt X_n(t)$ be two solutions to (\ref{3.3}) for $t\in
[0, \ff{nT}{n+1}].$ Since the support of $h_n$ is a compact subset
of $D$ and since $K_m\uparrow D$, there exists $m>n$ such that
$K_m\supset {\rm supp}\, h_n$. Then (\ref{3.2}) yields that
$$2\<b_n(t,x)-b_n(t,y), x-y\> +\hsn{\si_n(t,x)-\si_n(t,y)}^2
 \le C_n|x-y|^2 u(|x-y|^2)$$ holds for some constant $C_n>0,$   all $t\in [0, \ff{nT}{n+1}]$ and $x,y\in \R^d$ with $|x-y|\le 1.$ By the It\^o formula, this implies
\begin{equation}\label{3.4}
\begin{split}
\d |X_n(t)-\tt X_n(t)|^2 &\le C_n |X_n(t)-\tt X_n(t)|^2 u(|X_n(t)-\tt X_n(t)|^2) \d t\\& \ \ + 2 \<X_n(t) -\tt X_n(t), \{\si_n(t,X_n(t))
-\si_n(t, \tt X_n(t))\}\d B(t)\>
\end{split}
\end{equation} for
 $t\in [0, \ff{nT}{n+1}].$ On the other hand, $u\in \U$ implies that
$$u(r)+ru'(r)\ge \ll,\ \ \ r\in [0,\rho_0]$$ holds for some constants $\ll, \rho_0>0.$ Let
$$\Psi_\vv(r)= \exp\bigg[\ll\int_1^r\ff{\d s}{\vv+ su(s)}\bigg],\ \ r,\vv\ge 0.$$ Then, for any $\vv>0,$ we have $\Psi_\vv\in C^2([0,\infty))$ and
\beg{equation*}\beg{split} & r u(r) \Psi_\vv'(r) =\ff{\ll ru(r)}{\vv+ru(r)}\Psi_\vv(r)\le \ll\Psi_\vv (r),\\
& \Psi_\vv''(r)= \ff{\ll^2-\ll \{u(r)+ r u'(r)\}}{(\vv+ r u(r))^2} \le 0,\ \ r\in [0,\rho_0].\end{split}\end{equation*} Therefore, letting
$$\tau_0=\inf\Big\{t\in \Big[0, \ff{nT}{n+1}\Big]:\ |X_n(t)-\tt X_n(t)|^2 \ge \rho_0\Big\},$$ it follows from (\ref{3.4}) and the It\^o formula that
\beg{equation*}\beg{split} &\d\Psi_\vv(|X_n(t)-\tt X_n(t)|^2)
\le \ll C_n \Psi_\vv(|X_n(t)-\tt X_n(t)|^2) \d t \\
&\qquad+ 2 \Psi_\vv'(|X_n(t)-\tt X_n(t)|^2) \<X_n(t) -\tt X_n(t), \{\si_n(t,X_n(t))-\si_n(t, \tt X_n(t))\}\d B(t)\>\end{split}\end{equation*} holds for $t\le \tau_0\land   \ff{nT}{n+1}.$ Hence,
$$ \E \Psi_\vv(|X_n(t\land\tau_0)-\tt X_n(t\land\tau_0)|^2)\le \e^{\ll C_n t}\Psi_\vv(0),\ \ t\le  \ff{nT}{n+1}.$$ Letting $\vv\downarrow 0$ and noting that $\Psi_0(0)=0,$ we arrive at $$\E \Psi_0(|X_n(t\land\tau_0)-\tt X_n(t\land\tau_0)|^2)=0.$$ Thus, $X_n(t\land\tau_0)-\tt X_n(t\land\tau_0)$ holds  for all  $t\in [0, \ff{nT}{n+1}].$ Therefore, $\tau_0=\infty$ and $X_n(t)=\tt X_n(t)$ holds for all  $t\in [0, \ff{nT}{n+1}].$
In conclusion, for every $n\ge 1$, the equation (\ref{3.3}) has a unique solution up to time $ \ff{nT}{n+1}.$

Since $h_n=1$ on ${\bf K}_n$ so that (\ref{3.3}) coincides with (\ref{3.1}) before the solution leaves ${\bf K}_n$, the equation (\ref{3.1}) has a unique solution $X(t)$ up to the time
$$\zeta_n:= \ff{nT}{n+1}\land \inf\{t\ge 0: X(t)\notin {\bf K}_n\}.$$ Therefore, (\ref{3.1}) has a unique solution up to the life time $\zeta=T\land\lim_{n\to\infty} \zeta_n.$ \end{proof}

\subsection{Stochastic functional differential equations}

Let $\C:=\C([-r_0,0];\R^d)$ for a fixed number $r_0>0$, and for any set $A\subset \R^d$ let $A^\C= \{\phi\in\C: \phi([-r_0,0])\subset A\}.$   For fixed  $T>0$ and a non-empty open domain $D$ in $\R^d$, we consider  the SFDE
\beq\label{FE} \d X(t)= \bar b(t, X_t)\d t + \bar \si(t,X_t)\d B(t),\ \ X_0\in D^\C,\end{equation}
where $B(t)$ is the $m$-dimensional Brownian motion, $\bar  b: [0,T)\times  D^\C\to \R^d$ and $\bar\si: [0,T)\times D^\C\to \R^d\otimes \R^m$ are measurable, bounded on $[0,t]\times K^\C$ for $t\in [0,T)$ and compact set $K\subset D$, and continuous in the second variable.

\begin{thm}\label{T4.2} Assume that  there exists a sequence of compact sets $\K_n\uparrow D$ such that for every $n\ge 1$,
\beq\label{FE2} 2\<\bar b(t, \phi)-\bar b(t,\psi), \phi(0)-\psi(0)\>
+\hsn{\bar\si(t, \phi)-\bar\si(t,\psi)}^2\le \|\phi-\psi\|_\infty^2
u_n(\|\phi-\psi\|_\infty^2)\end{equation} and \beq\label{FF3}
\hsn{\bar \si(t,\phi)-\bar\si(t,\psi)}^2 \le
\|\phi-\psi\|_\infty^2u_n(\|\phi-\psi\|_\infty^2)\end{equation} hold
for some $u_n\in\bar{\scr U}$ and all $\phi,\psi\in \K_n^\C, t\le
\ff{nT}{n+1}.$ Then for any initial data $X_0\in D^\C$, the equation
$(\ref{FE})$ has a unique solution $X(t)$ up to life time
$$\zeta:= T\land \lim_{n\to\infty} \inf\big\{t\in [0,T):\ X(t)\notin \K_n\big\}.$$ \end{thm}

\beg{proof} Using the approximation argument in the proof of Theorem \ref{T3.1}, we may and do assume that $D=\R^d$ and  $a$ and $\bar\si$ are bounded and continuous in the second variable and prove the existence and uniqueness of solution up to any time $T'<T$. According to the Yamada-Watanabe principle, we shall verify below the existence of a weak solution and the pathwise uniqueness of the strong solution respectively.

(1) The proof of the existence of a weak solution is standard up to an approximation argument.   Let $\mathbf B(s)= B(r_0+1+s), s\in [-r_0,0],$ where $B(s)$ is a $d$-dimensional Brownian motion.  Define
$$\bar\si_n(t,\phi)= \E \bar\si(t, \phi+n^{-1}\mathbf B),\bar b_n(t,\phi)= \E \bar b(t, \phi+n^{-1}\mathbf B),\ \ n\ge 1.$$ Applying \cite[Corollary 1.3]{BWY} for $\si=\ff 1 nI_{d\times d}, m=0, Z=b=0$ and $T= 1+r_0$, we see that for every $n\ne 1$, $\bar\si_n$ and $\bar b_n$ are Lipschitz continuous in the second variable uniformly in the first variable. Therefore, the equation
$$\d X^{(n)}(t)= \bar b_n(t, X^{(n)}_t)\d t + \bar \si_n(t,X^{(n)}_t)\d B(t),\ \ X_0^{(n)}=X_0$$ has a unique strong solution up to time $T'$: $X^{(n)}\in C([0,T'];\R^d).$  To see that $X^{(n)}$ converges weakly as $n\to\infty$, we take the reference function
$$g_\vv(h):= \sup_{ t\in [0,T)} \sup_{s\in (0, (T-t)\land 1)} \ff{|h(t+s))-h(t)|}{s^\vv}$$ for a fixed number $\vv\in (0,\ff 1 2).$ It is well known that $g_\vv$ is a compact function on $C([0,T'];\R^d)$, i.e. $\{g_\vv\le r\}$ is compact under the uniform norm for any $r>0.$ Since $\bar b_n$ and $\bar\si_n$ are uniformly bounded and $\vv\in (0,\ff 1 2)$, we have
$$\sup_{n\ge 1} \E g_\vv(X^{(n)}) <\infty.$$ Let $\P^{(n)}$ be the distribution of $X^{(n)}$. Then the family $\{\P^{(n)}\}_{n\ge 1}$ is tight, and hence (up to a sub-sequence) converges weakly to a probability measure $\P$ on $\OO:= C([0,T;];\R^d)$. Let $\F_t=\si(\oo\mapsto \oo(s): s\le t)$ for $t\in [0,T'].$ Then the coordinate process
$$X(t)(\oo):= \oo(t),\ \ t\in [0,T'], \oo\in \OO$$ is $\F_t$-adapted. Since $\P^{(n)}$ is the distribution of $X^{(n)}$, we see that
$$M^{(n)}(t):= X(t)-\int_0^t \bar b_n(s, X_s)\d s,\ \ t\in [0,T']$$ is a $\P^{(n)}$-martingale with
$$\<M^{(n)}_i, M_j^{(n)}\>(t) =\sum_{i=1}^m \int_0^t \big\{(\bar\si_n)_{ik}(\bar\si_n)_{jk}\big\}(s,X_s)\d s,\ \ 1\le i,j\le d.$$ Since $\bar\si_n\to\bar\si$ and $\bar b_n\to \bar b$ uniformly and $\P^{(n)}\to \P$ weakly, by letting $n\to \infty$ we conclude that
$$M(t):= X(t)-\int_0^t \bar b(s,X_s)\d s,\ \ s\in [0,T']$$ is a $\P$-martingale with
$$\<M_i, M_j\>(t) =\sum_{i=1}^m \int_0^t \big\{\bar \si_{ik}\bar\si_{jk}\big\}(s,X_s)\d s,\ \ 1\le i,j\le d.$$ According to \cite[Theorem II.7.1]{IW}, this implies   $$M(t)= \int_0^t \bar\si(s,X_s)\d B(s),\ \ t\in [0,T']$$ for some $m$-dimensional Brownian motion $B$ on the filtered probability space $(\OO, \F_t, \P).$
 Therefore, the equation has a weak solution up to time $T'.$

(2) The pathwise uniqueness. Let $X(t)$ and $Y(t)$ for $t\in [0,T']$ be two strong solutions with $X_0=Y_0$. Let $Z=X-Y$ and
$$\tau_n= T'\land   \inf\big\{t\in [0,T):\ |X(t)|+|Y(t)|\ge n\big\}.$$ By the It\^o formula and (\ref{FE2}), we have
\beq\label{ABC}\d |Z(t)|^2\le 2\<(\bar \si(t, X_t)-\bar\si(t,Y_t))
\d B(t), Z_t\> + \|Z_t\|_\infty^2u_n(\|Z_t\|_\infty^2),\ \ t\le
\tau_n.\end{equation} Let
$$\ell_n(t):=\sup_{s\le t\land\tau_n} |Z_s|^2,\ \ t\ge 0.$$
Noting that $su_n(s)$ is increasing in $s$, we have
$$\|Z_t\|_\infty^2u_n(\|Z_t\|_\infty^2)\le \ell_n(t)
u_n(\ell_n(t)),\ \ t\ge 0.$$ So, by (\ref{FF3}), (\ref{ABC}) and
using the Burkholder-Davis-Gundy inequality, there exist constants
$C_1,C_2>0$ such that
\beg{equation*}\beg{split} \E\ell_n(t) &\le \int_0^t \E \ell_n(s) u_n(\ell_n(s))\d s + C_1 \E \bigg(\ell_n(t)\int_0^t\ell_n(s)u_n(\ell_n(s))\d s\bigg)^{1/2}\\
&\le \ff 1 2 \E \ell_n(t) + C_2 \int_0^t \E \ell_n(s) u_n(\ell_n(s))\d s.\end{split}\end{equation*} Since $s\mapsto su_n(s)$ is concave, due to Jensen's inequality this implies that
$$\E \ell_n(t)\le 2 C_2\int_0^t \E \ell_n(s) u_n\big(\E\ell_n(s)\big)\d s.$$ Let $G(s)= \int_1^s\ff 1 {s u_n(s)}\d s,\ s>0,$ and let $G^{-1}$ be the inverse of $G$. Since $\int_0^1\ff 1 {su_n(s)}\d s=\infty$, we have $[-\infty,0]\subset \text{Dom}(G^{-1})$ with $G^{-1}(-\infty)=0$. Then,   by the Bihari's inequality (cf. \cite[Theorem 1.8.2]{Mao}), we obtain
$$\E\ell_n(t) \le G^{-1} \big(G(0)+ G(2C_2 t)\big)=G^{-1}(-\infty)=0.$$ This implies that $X(t)=Y(t)$ for $t\le \tau_n$ for any $n\ge 1$. Since $\bar b$ and $\bar\si$ are bounded, we have $\tau_n\uparrow T'$. Therefore, $X(t)=Y(t)$ for $t\in [0,T'].$
\end{proof}

\end{document}